\documentclass[a4paper,12pt]{amsart}
\usepackage[T1]{fontenc}    %sajat
\usepackage[utf8]{inputenc}
\usepackage{enumerate}

\def\ZZ{\mathbb Z}

\newcommand{\set}[1]{\left\lbrace #1\right\rbrace}%set

\newcommand{\remove}[1]{ }

\newcommand{\lnko}[1]{\langle #1\rangle}%pgcd

\newtheorem{theorem}{Theorem}[section]
\newtheorem{proposition}[theorem]{Proposition}
\newtheorem{lemma}[theorem]{Lemma}
\newtheorem{corollary}[theorem]{Corollary}

\theoremstyle{remark}
\newtheorem*{remark}{Remark}
\newtheorem*{remarks}{Remarks}\remove{}
\newtheorem*{example}{Example}
\newtheorem*{examples}{Examples}

\numberwithin{equation}{section}

\begin{document}
\title{Positive orthogonal functions}
\author{Florian Delage} 
\address{Département de mathématique\\
         Université de Strasbourg\\
         7 rue Re\-né Descartes\\
         67084 Strasbourg Cedex, France}
\email{fdelage@math.unistra.fr}
\thanks{}
\subjclass[2010]{42A75}
\keywords{Almost periodic functions, orthogonality, oscilations, residue classes}
%\curraddr{}
\date{Version of 2015-10-16}
%\dedicatory

\begin{abstract}
The existence or non-existence of positive orthogonal functions for subspaces of almost periodic functions has important applications in studying the oscillatory behavior of vibrations. 
Cazenave, Haraux and Komornik have obtained a number of theorems of this type. 
The purpose of this paper is to answer an open question formulated in the 1980's, and to completely clarify the situation for subspaces defined by three periods.
\end{abstract}

\maketitle

\section{Introduction}\label{s1}

Given a positive real number $T$, we denote by $X_T$ the vector space of locally square summable, $T$-periodic functions of zero mean value. 
We will often identify the functions in $X_T$ with their restrictions to the interval $(0,T)$.

By the elements of Fourier series we know that
\begin{itemize}
\item $X_T$ is spanned by the family of exponential functions $e^{2i\pi mt/T}$ where $m$ runs over the set of non-zero integers;
\item $X_T$ is the orthogonal complement in $L^2(0,T)$ of the characteristic function $h_1$  of $(0,T)$.
\end{itemize}
We observe in this case that $h_1$ is strictly positive on $(0,T)$.

Now let us consider finitely many positive real numbers $T_1,\ldots, T_n$ and the vector space  $X:=X_{T_1}+\cdots+X_{T_n}$. 
Generalizing some earlier theorems in \cite{HarKom1984} and \cite{HarKom1985}, the following results were obtained in \cite{Kom1987}:
\begin{itemize}
\item there exists a  number $T>0$ such that $X$ is dense in $L^2(0,T')$ for all $0<T'<T$, and $X$ is not dense for any $T'>T$;
\item $X$ has a finite codimension $d$ in $L^2(0,T)$, satisfying the inequalities $1\leq d \leq n$.
\end{itemize}
Moreover, $T$ and $d$ have been explicitly determined in function of the arithmetical properties of the periods $T_1,\ldots, T_n$.

In particular, it was shown that $d=1$ if and only if all the ratios $T_i/T_j$ are rational. 
In this case $X_T$ is the orthogonal complement in $L^2(0,T)$ of some function $h_n$  again.
Since $X$ is invariant for complex conjugation, $h_n$ may (and will) be assumed to be real-valued. 
Then we may ask whether this function may be chosen to be strictly positive on $(0,T)$ as in case $n=1$. 

Apart from its own interest, the existence or non-existence of positive orthogonal functions has important consequences concerning certain oscillation properties of vibrating membranes and plates; see, e.g., \cite{CazHar1984}, \cite{CazHar1985},  \cite{CazHar1987}, \cite{CazHar1988}, \cite{HarKom1985}, \cite{HarKom1984},  \cite{Kom1989a}, \cite{Kom1990}, \cite{Kom1989c},  \cite{HarKom1991}, \cite{Kom1993} and their references.

It follows from the results of \cite{Kom1987} that $h_2$ always may be chosen to be strictly positive, while for $n=4$ and $(T_1,T_2,T_3,T_4)=(105,70,42,30)$ the function $h_4$ has both strictly positive and strictly negative values. 

It remained an open question whether there exist such counterexamples for $n=3$. 
Answering this question we prove that $h_3$ can never have both strictly positive and strictly negative values. 

Moreover, we completely clarify the sign of $h_3$, by characterizing its strict positivity in functions of the arithmetic properties of the periods $T_1,T_2,T_3$.

Since the ratios $T_i/T_j$ are rational, we may assume by scaling that the periods are positive integers. 
We will use the notations $\lnko{a,b}$ and $\lnko{a,b,c}$ for the greatest common divisor of the integers $a,b$ and $a,b,c$, respectively. 

In this case the results of \cite{Kom1987} show that 
\begin{equation*}
T=T_1+T_2+T_3-\lnko{T_1,T_2}-\lnko{T_2,T_3}-\lnko{T_3,T_1}+\lnko{T_1,T_2,T_3},
\end{equation*}
and that $h_3$ is constant on each interval
\begin{equation*}
((j-1)\lnko{T_1,T_2,T_3}, j\lnko{T_1,T_2,T_3}),\quad j=1,\ldots, T/\lnko{T_1,T_2,T_3}.
\end{equation*} 
It follows from the minimality of $T$ that this constant is nonzero on the first and the last of these intervals.

Denoting the sequence of these constants by $c_0,\ldots, c_{T/\lnko{T_1,T_2,T_3}}$, we will often write 
\begin{equation*}
h_n\sim c_0,\ldots, c_{T/\lnko{T_1,T_2,T_3}}
\end{equation*}
for brevity.
Then $c_0\ne 0$ and $c_{T/\lnko{T_1,T_2,T_3}}\ne 0$.

\begin{theorem}\label{t11}
Let $T_1$, $T_2$, $T_3$ be three positive integers, and choose $h_3$ to be positive in a right neighborhood of $0$.

\begin{enumerate}[\upshape (i)]
\item $h_3$ is nonnegative in $(0,T)$.
\item $h_3$ vanishes on at least one of the above intervals if and only if $1< \lnko{T_i,T_j}< \min(T_i,T_j)$ whenever $i\ne j$.
\end{enumerate}
\end{theorem}

\begin{examples}\label{e12}
~
\begin{itemize}
\item For $(T_1,T_2,T_3)=(26, 24, 9)$ we have $\lnko{26,9}=1$ and
\begin{center}
$h_3\sim $ 1, 2, 3, 5, 7, 9, 12, 15, 18, 21, 24, 27, 29, 31, 33, 34, 35, 36, 36, 36, 36, 36, 36, 36, 36, 36, 
 35, 34, 33, 3, 29, 27, 24, 21, 18, 15, 12, 9, 7, 5, 3, 2, 1,
\end{center}
so that $h_3>0$.
\item For $(T_1,T_2,T_3)=(4, 8, 13)$ we have $\lnko{4,8}=\min(4,8)$ and
\begin{center}
$h_3\sim $ 1, 1, 1, 1, 1, 1, 1, 1, 2, 2, 2, 2, 2, 2, 2, 2, 3, 3, 3, 3, 3, 3, 3, 3, 4, 4, 4, 4, 4, 4, 4, 4, 4, 4, 4, 4, 4, 4, 4, 4, 3, 3, 3, 3, 3, 3, 3, 3, 2, 2, 2, 2, 2, 2, 2, 2, 1, 1, 1, 1, 1, 1, 1, 1,
\end{center}
so that $h_3>0$ again.
\item On the other hand, in case $(T_1,T_2,T_3)=(35, 21, 15)$ we have $\lnko{35,21}=7$, $\lnko{35,15}=5$ and $\lnko{21,15}=3$.
Hence $1< \lnko{T_i,T_j}< \min(T_i,T_j)$ whenever $i\ne j$, and
\begin{center}
$h_3\sim $ 1, 0, 0, 1, 0, 1, 1, 1, 1, 1, 2, 1, 2, 2, 2, 2, 2, 3, 2, 3, 3, 2, 3, 3, 3, 3, 3, 3, 3, 3, 3, 3, 3, 3, 3, 2, 3, 3, 2, 3, 2, 2, 2, 2, 2, 1, 2, 1, 1, 1, 1, 1, 0, 1, 0, 0, 1.
\end{center}
\end{itemize}
\end{examples}

The proof of Theorem \ref{t11} will be based on the following explicit representation of the function $h_3$: 

\begin{theorem}\label{t12}
There exist three positive integer $p,q,r$ and a sequence $(a_j)_{j=0}^{pqr-q-r+1}$ of real numbers such that 
\begin{equation*}
h_3(t)=\sum_{j=0}^{pqr-q-r+1}a_jh_2(t-j\lnko{T_1,T_2,T_3}).
\end{equation*}
Furthermore,
\begin{enumerate}[\upshape (i)]
\item if $q=1$, then $(a_j)_{j=0}^{pqr-q-r+1}=(a_j)_{j=0}^{pr-r}$ is the beginning of the sequence $(1,0^{r-1})^{\infty}$;
\item if $r=1$, then $(a_j)_{j=0}^{pqr-q-r+1}=(a_j)_{j=0}^{pq-q}$ is the beginning of the sequence $(1,0^{q-1})^{\infty}$;
\item if $q\ge 2$ and $r\ge 2$, then the sequence $(a_j)_{j=0}^{pqr-q-r+1}$ is $qr$-periodical and it is the beginning of the sequence 
\begin{equation*}
\sum_{l=0}^{\infty}0^{lq}(1,-1,0^{r-2})^{\infty}=\sum_{l=0}^{\infty}0^{lr}(1,-1,0^{q-2})^{\infty}.
\end{equation*}
\end{enumerate}
\end{theorem}
In the above theorem we use the following notations:
\begin{align*}
&(1,0^{r-1})=1,\underbrace{0,\ldots \ldots,0}_{r-1};\\
&(1,-1,0^{3})^{\infty}=1,-1,0,0,0,1,-1,0,0,0,1,-1,0,0,0,\ldots .
\end{align*}

The plan of this paper is the following. 
In Sections \ref{s2}, \ref{s3} we recall the construction of $h_n$ for $n=2, 3$, and we present a new, simpler construction of $h_3$.
Using this construction Theorems \ref{t11} and \ref{t12} are proved in Sections \ref{s4}--\ref{s6}. 
In the final Section \ref{s7} of the paper we present some partial results and conjectures concerning the case $n>3$ and the existence of a smallest number $T'$  such that $L^2(0,T')$ contains a strictly positive function orthogonal to $X$.

%%%%%%%%%%%%%%%%%%%%%%%%%%%%%%%%%%%%%%%%%

\section{The construction of $h_2$}\label{s2}

Fix two positive integers $T_1$ and $T_2$.
Then the set $X:=\overline{X_{T_1}+X_{T_2}}$ is spanned by the functions $e^{2i\pi mt/T_1T_2}$ where $m$ runs over the non-zero multiples of $T_1$ and $T_2$.
 
We recall from \cite{Kom1987} that
\begin{equation*}
T_{12}:= T=T_1+T_2-\lnko{T_1,T_2}.
\end{equation*} 

\begin{proposition}\label{p21}
Apart from a multiplicative constant the function $h_2$ is given by the formula
\begin{equation}\label{21}
h_2(t):=\sum_{j=0}^{\frac{T_{12}-T_1}{\lnko{T_1,T_2}}}h_1(t-j\lnko{T_1,T_2})=\sum_{j=0}^{\frac{T_2}{\lnko{T_1,T_2}}-1}h_1(t-j\lnko{T_1,T_2}),
\end{equation}
where $h_1$ denotes the characteristic function of the interval $[0,T_1)$.

In particular, $h_2>0$ in $(0,T)$.
\end{proposition}

\begin{remark}
If $T_2$ is a divisor of $T_1$, then we have simply $X=X_{T_1}$, $T_{12}=T_1$ and  $h_2=h_1$.
\end{remark}

\begin{proof}
We assume by scaling that $\lnko{T_1,T_2}=1$.
Since we already know that $h_2$ is unique apart from a multiplicative constant, it suffices to check that the function given by \eqref{21} satisfies the equality
\begin{equation*}
\int_0^Th_2(t)e^{2i\pi mt/T_1T_2}dt=0
\end{equation*}
for all non-zero integers $m$ such that $m/T_1$ or $m/T_2$ is integer.

Using \eqref{21} we obtain
\begin{align*}
\int_0^Th_2(t)e^{2i\pi mt/T_1T_2}dt 
&= \int_0^T\sum_{j=0}^{T_2-1}h_1(t-j)e^{2i\pi mt/T_1T_2}dt\\
&= \sum_{j=0}^{T_2-1} e^{2i\pi mj/T_1T_2}\int_0^Th_1(t)e^{2i\pi mt/T_1T_2}dt.
\end{align*}

If  $k:=m/T_2$ is a non-zero integer and, then, since $h_1=0$ in $(T_1,T)$,
\begin{equation*}
\int_0^Th_1(t)e^{2i\pi mt/T_1T_2}dt=\int_0^{T_1}h_1(t)e^{2i\pi k t/T_1}dt=0
\end{equation*}
by the definition of $h_1$.

If $\ell:=m/T_1$ is  integer and  $m/T_2$ is not integer, then $\ell /T_2=m/T_1T_2$ is not  integer either, so that
\begin{equation*}
\sum_{j=0}^{T_2-1} e^{2i\pi mj/T_1T_2}=\sum_{j=0}^{T_2-1} e^{2i\pi \ell j/T_2}=\frac{e^{2i\pi \ell}-1}{e^{2i\pi\ell /T_2 }-1}=0.\qedhere
\end{equation*}
\end{proof}
 
\section{The construction of $h_3$}\label{s3}

Fix three positive integers $T_1$, $T_2$ and $T_3$.
Then  the set
\begin{equation*}
X:=\overline{X_{T_1}+X_{T_2}+X_{T_3}}
\end{equation*} 
is spanned by the functions $e^{2i\pi mt/T_1T_2T_3}$ where $m$ runs over the non-zero multiples of $T_1T_2$, $T_2T_3$ and $T_3T_1$.
 
We recall from \cite{Kom1987} that
\begin{equation*}
T_{123}:= T=T_1+T_2+T_3-\lnko{T_1,T_2}-\lnko{T_2,T_3}-\lnko{T_3,T_1}+\lnko{T_1,T_2,T_3}.
\end{equation*} 

In order to simplify the formulas we assume henceforth by a scaling argument that 
\begin{equation}\label{31} 
\lnko{T_1,T_2,T_3}=1.
\end{equation}

We will prove that $h_3$, which is determined apart from a multiplicative constant, has the form 
\begin{equation*}   
h_3(t)=\sum_{j=0}^{T_{123}-T_{12}}a_jh_2(t-j)
\end{equation*}
with $h_2$ and $T_{12}$ as introduced in the preceding section, and  a suitable sequence $(a_j)$ of real numbers. 
The sign of $h_3$ will depend on the coefficients $a_j$.

The coefficients $a_j$ have to be chosen so that
\begin{equation}\label{32}
\int h_3(t)e^{2i\pi mt/(T_1T_2T_3)} dt=0
\end{equation}
for all non-zero integers $m$ such that at least one of the fractions $m/T_1T_2$, $m/T_1T_3$ or $m/T_2T_3$  is integer. 

Let us rewrite the integral in \eqref{32} as follows:
\begin{align*}
\int h_3(t)&e^{2i\pi mt/(T_1T_2T_3)} dt\\
&=\sum_{j=0}^{T_3-\lnko{T_1,T_3}-\lnko{T_2,T_3}+1}a_j\int h_2(t-j)e^{2i\pi mt/(T_1T_2T_3)} dt\\ 
&=\sum_{j=0}^{T_3-\lnko{T_1,T_3}-\lnko{T_2,T_3}+1}a_je^{2i\pi mj/(T_1T_2T_3)}\int h_2(t)e^{2i\pi mt/(T_1T_2T_3)} dt.
\end{align*}
Using the properties of $h_2$ the last integral  vanishes for all non-zero multiples of  $T_1T_3$ and $T_2T_3$. 
Therefore it suffices to choose the coefficients $a_j$ so that
\begin{equation}\label{33}
\sum_{j=0}^{T_3-\lnko{T_1,T_3}-\lnko{T_2,T_3}+1}a_je^{2i\pi mj/(T_1T_2T_3)}=0
\end{equation}
for all multiples $m$ of $T_1T_2$ such that none of $m/T_1T_3$ and $m/T_2T_3$ is integer.

The following remarks allow us to further simplify the problem.

\begin{enumerate}[\upshape (i)]
\item Using the periodicity of the exponential functions the equations \eqref{33} do not change if we add to $m$ a multiple of $T_1T_2T_3$.
We may therefore assume that $0<m<T_1T_2T_3$.
\item If $T_3$ is a divisor of $T_1$ or $T_2$ then $X=X_{T_1}+X_{T_2}$, $T_{123}=T_{12}$ and $h_3=h_2$. In this case $h_3>0$ on $(0,T)$.
\item If $\lnko{T_1,T_3}=\lnko{T_2,T_3}=1$, then we may choose $a_j=1$ for all $j$. Hence $h_3>0$ on $(0,T)$ again.
\end{enumerate}

In order to investigate the general cases we transform the system \eqref{33} further as follows.

Setting $n:=m/T_1T_2$, the system  becomes
\begin{equation}\label{34}
\sum_{j=0}^{T_3-\lnko{T_1,T_3}-\lnko{T_2,T_3}+1}a_je^{2i\pi nj/T_3}=0, \quad n\in N,
\end{equation}
where $N$ is the set of integers $n$ satisfying
\begin{equation}\label{35}
0<n<T_3, \quad \frac{nT_2}{T_3}\notin \ZZ \quad \text{and} \quad \frac{nT_1}{T_3}\notin \ZZ.
\end{equation}

Let us introduce the  integers
\begin{equation}\label{36}
q:=\lnko{T_1,T_3}\quad\text{and}\quad r:=\lnko{T_2,T_3}.
\end{equation}
Then $\lnko{q,r}=1$ by our assumption $\lnko{T_1,T_2,T_3}=1$, and therefore 
\begin{equation}\label{37}
T_3=pqr
\end{equation}
for a suitable positive integer $p$.

Let us also introduce two positive integers $\alpha$ and $\beta$ such that 
\begin{equation*}
T_1=\alpha q \quad \text{and} \quad T_2=\beta r;
\end{equation*}
then we have 
\begin{align*}
q&=\lnko{T_1,T_3}=\lnko{\alpha q, pqr}=q\lnko{\alpha , pr}
\intertext{and} 
r&=\lnko{T_2,T_3}=\lnko{\beta r, pqr}=r\lnko{\beta,pq}.
\end{align*}
We infer from the above  that 
\begin{equation}\label{38}
\lnko{q,r}=\lnko{\alpha , pr}=\lnko{\beta,pq}=1.
\end{equation}

Using these relations the expressions \eqref{34}--\eqref{35} may be rewritten in the form
\begin{equation}\label{39}
\sum_{j=0}^{pqr-q-r+1}a_je^{2i\pi jn/pqr}=0, \quad n\in N,
\end{equation}
where $N$ is the set of integers satisfying the conditions 
\begin{equation*}
0<n<pqr, \quad \frac{n\beta}{pq}\notin \ZZ \quad \mbox{and} \quad \frac{n\alpha}{pr}\notin \ZZ.
\end{equation*}
Using \eqref{38} the definition of $N$ may be simplified to
\begin{equation}\label{310}
N:=\set{n\in\ZZ\ :\ 0<n<pqr, \quad \frac{n}{pq}\notin \mathbb{Z} \quad \mbox{and} \quad \frac{n}{pr}\notin \mathbb{Z}}.
\end{equation}

Now we are ready to prove the existence of $h_3$ in the required form, and even to construct the coefficients $a_j$ explicitly:

\begin{proposition}\label{p31}
There exists a sequence $(a_j)$ of real numbers satisfying \eqref{39}.

Moreover, the following identify holds: 
\begin{equation}\label{311}
\sum_{j=0}^{pqr-q-r+1}a_jx^j=\frac{(1-x)(1-x^{pqr})}{(1-x^q)(1-x^r)}.
\end{equation} 
\end{proposition}

\begin{proof}
First we show that the right side of \eqref{311} defines a polynomial $Q(x)$. 
For this we prove that each zero of the denominator is also a zero of the numerator, with at least the same multiplicity.

It is clear that $x=1$ is a double zero of both the denominator and the numerator. 

All other zeros of the denominator are simple because of the condition $\lnko{q,r}=1$ (see \eqref{38}), and all of them are also zeros of the factor $1-x^{pqr}$ in the numerator. 

Next we observe that if $n$ satisfies the conditions \eqref{310}, then $x:=e^{2i\pi n/pqr}$ is a zero of the factor $1-x^{pqr}$ in the numerator, but it is not a zero of the denominator. 
Hence $Q(x)=0$.
We conclude by writing
\begin{equation*}
Q(x)=\sum_{j=0}^{pqr-q-r+1}a_jx^j.\qedhere
\end{equation*}
\end{proof}

\section{Proof of the first part of Theorem \ref{t11}}\label{s4}

Given a step function $h$ of the form 
\begin{equation*} 
h(x):=c_k \quad \mbox{for}\quad x\in [k,k+1), \quad k\in\ZZ
\end{equation*}
with $c_k=0$ for all but finitely many nonnegative integers $k$, we introduce an associated polynomial by setting 
\begin{equation*} 
P(x):=\sum_kc_kx^k. 
\end{equation*}
We observe that if the coefficients of P are nonnegative, then $h$ is a nonnegative function. 

We know that $h_1, h_2, h_3$ are step functions. 
Let us determine the polynomials $P_1, P_2, P_3$ associated with $h_1, h_2, h_3$.
\begin{enumerate}[\upshape (i)]
\item We have $h_1=\chi_{[0,T_1)}$ and hence
\begin{equation*}
P_1(x)=\sum_{k=0}^{T_1-1}x^k=\frac{1-x^{T_1}}{1-x}.
\end{equation*}

\item It follows from the formula 
\begin{equation*}
h_2(t)=\sum_{j=0}^{\frac{T_2}{\lnko{T_1,T_2}}-1}h_1(t-j\lnko{T_1,T_2})
\end{equation*}
that 
\begin{equation*}
P_2(x)=\sum_{j=0}^{\frac{T_2}{\lnko{T_1,T_2}}-1}x^{j\lnko{T_1,T_2}}\cdot P_1(x)=\frac{1-x^{T_2}}{1-x^{\lnko{T_1,T_2}}}\cdot \frac{1-x^{T_1}}{1-x}.
\end{equation*}

\item Applying Proposition \ref{p31} it follows from the formula 
\begin{equation*}
h_3(t)=\sum_{j=0}^{T_{123}-T_{12}}a_jh_2(t-j)
\end{equation*}
that 
\begin{align*}
P_3(x)
&=\sum_{j=0}^{pqr-q-r+1}a_jx^j P_2(x)=Q(x)P_2(x)\\
&=\frac{1-X^{pqr}}{1-X^q}\cdot \frac{1-X^{T_2}}{1-X^r} \cdot \frac{1-X^{T_1}}{1-X^{\lnko{T_1,T_2}}}.
\end{align*}
\end{enumerate}
We complete the proof by observing that each fraction of the last product is a polynomial with nonnegative coefficients. 
Indeed, since $pqr, T_2, T_1$ are multiples of $q, r, \lnko{T_1,T_2}$, respectively, this follows from the identity 
\begin{equation*}
\frac{1-X^{ab}}{1-X^b}=\sum_{k=0}^ax^{kb},
\end{equation*}
valid for all positive integers $a$ and $b$.

We have established the nonnegativity of $h_3$.
For the proof of Theorem \ref{t11} (ii) we need the more explicit formula of Theorem \ref{t12} of the coefficients $a_j$. 
This will be derived in the next section by a different method.

\section{Proof of Theorem \ref{t12}}\label{s5}

We continue to use the notations and assumptions of Section \ref{s3}; in particular we assume that $\lnko{T_1,T_2,T_3}=1$.

\begin{proof}[Proof of Theorem \ref{t12} (i), (ii)]
In case $q=1$ we have to show that if $(a_j)_{j=0}^{pr-r}$ is the beginning of the sequence $(1,0^{r-1})^{\infty}$ and $A=e^{2i\pi n/pqr}$ for some $n\in N$ (see \eqref{310}), then 
\begin{equation*} 
\sum_{j=0}^{pr-r}a_jA^j=0.
\end{equation*}
Since $A^{pqr}=1$ and $A^r\ne 1$ (because $\dfrac{n}{pq}\notin \mathbb{Z}$), we have
\begin{equation*}
\sum_{j=0}^{r(p-1)}a_jA^j= 1+A^r+\cdots+A^{(p-1)r}=\frac{1-A^{pr}}{1-A^r}=0.
\end{equation*}

In case $r=1$ we may repeat this proof by exchanging $q$ and $r$.
\end{proof}

For the proof of Theorem \ref{t12} (iii) we need a lemma:

\begin{lemma}\label{l51}\mbox{}

\begin{enumerate}[\upshape (i)]
\item For each $0<j<qr$ the sum defining $a_j$ contains at most one term $1$ and at most one term $-1$. 
\item If $(q-1)(r-1)<j<qr$, then we have $a_j=0$. 
\item We have $a_{(q-1)(r-1)}=1$.
\end{enumerate}
\end{lemma}

\begin{proof}
\begin{enumerate}[\upshape (i)]
\item Assume that the sum  defining $a_{j_0}$ contains two different terms $1$, i.e.,
\begin{equation*} 
j_0=l_1q+m_1r=l_2q+m_2r
\end{equation*}
with  some nonnegative integers $l_1,l_2,m_1,m_2$ such that $(l_1,m_1)\neq (l_2,m_2)$. 
Then $l_1\ne l_2$, $m_1\ne m_2$, $l_{1},l_{2}\in \{0,\dots, r-1\}$ because $0<j_0<qr$ and
\begin{equation*} 
(l_1-l_2)q=(m_2-m_1)r. 
\end{equation*} 
Since $q$ and $r$ are relatively prime, this implies that $r$ is a divisor of $(l_1-l_2)$, hence $|l_{1}-l_{2}|\geq r$ contradicting the hypothesis above.

Now assume that the sum  defining $a_{j_0}$ contains two different terms $-1$. 
Then we have 
\begin{equation*} 
j_0-1=l_1q+m_1r=l_2q+m_2r
\end{equation*}
with  some nonnegative integers $l_1,l_2,m_1,m_2$ such that $(l_1,m_1)\neq (l_2,m_2)$, and we may repeat the preceding proof.

\item If $j<qr$, then the sum  defining $a_j$ contains at most one term $1$ and at most one term $-1$ by (i). 
To complete the proof we show that for each index satisfying $(q-1)(r-1)<j<qr$ there exist nonnegative integers $l_1,l_2,m_1,m_2$ such that
\begin{equation*}  
j=l_1q+m_1r \quad \text{and} \quad  j=l_2q+m_2r+1. 
\end{equation*} 

Since $q$ and $r$ are relatively prime
\begin{equation*}
\{j-kq : 0\leq k\leq r-1\}
\end{equation*}
form a complete system of residues modulo $r$. 
Hence one of them, say $j-l_1q$ is a multiple of $r$, so that $j=l_1q+m_1r$ with some integer $m_1$.
Since $j>(q-1)(r-1)$, we have
\begin{equation*}
j-l_1q>(q-1)(r-1)-(r-1)q\ge 1-r,
\end{equation*}
so that $m_1\ge 0$. 

We remark for later reference that this part of the proof remains valid for $j=(q-1)(r-1)$, too, because then we still have
\begin{equation*}
j-l_1q\ge (q-1)(r-1)-(r-1)q\ge 1-r.
\end{equation*}

Similarly, the numbers
\begin{equation*}
\{j-1-kq : 0\leq k\leq r-1\}
\end{equation*}
form a complete system of residues modulo $r$, so that $j-1=l_2q+m_2r$ with some integer $0\le l_2\le r-1$ and $m_2$. 
Since
\begin{equation*}
j-1-l_1q\ge (q-1)(r-1)-(r-1)q\ge 1-r,
\end{equation*}
we have $m_2\ge 0$.

\item In view of (i) it sufficient to show that there exist nonnegative integers $l_1, m_1$ satisfying
\begin{equation*}  
(q-1)(r-1)=l_1q+m_1r, 
\end{equation*} 
but there do not exist nonnegative integers $l_2, m_2$ satisfying
\begin{equation*} 
(q-1)(r-1)=l_2q+m_2r+1. 
\end{equation*} 

The existence of $l_1$ and $m_1$ for $(q-1)(r-1)$ has been proved during the proof of  (ii).

Next, since 
\begin{equation*}
(r-1)q+1>(q-1)(r-1),
\end{equation*}
the second equality could only hold if $0\le l_2\le r-2$, i.e., if one of the numbers 
\begin{equation*}
(q-1)(r-1)-1-lq\quad (l=0,\ldots, r-2)
\end{equation*}
is a multiple of $r$. 
But this is not the case because  
\begin{equation*}
(q-1)(r-1)-1-lq\quad (l=0,\ldots, r-1)
\end{equation*}
is a complete system of residues modulo $r$, and the extra element 
\begin{equation*}
(q-1)(r-1)-1-(r-1)q=-r
\end{equation*}
is a multiple of $r$. 
\end{enumerate}
\end{proof}

Now we are ready to complete the proof of Theorem \ref{t12}. 

\begin{proof}[Proof of Theorem \ref{t12} (iii)]
Set again $A=e^{2i\pi n/pqr}$ for some $n\in N$ (see \eqref{310}).
We have clearly $A^{pqr}=1$.

Since
\begin{equation*}
pqr-q-r+1=(p-1)qr+(q-1)(r-1),
\end{equation*}
it follows from Lemma \ref{l51} (ii) and the $qr$-periodicity of $(a_j)$ that 
\begin{align*}
\sum_{j=0}^{pr-1-r+1}a_jA^j &= \sum_{k=0}^{p-1}\sum_{j=kqr}^{kqr+(q-1)(r-1)}a_jA^j\\
&= \sum_{k=0}^{p-1}\sum_{j=0}^{(q-1)(r-1)}a_{j+kqr}A^{j+kqr}\\
&= \sum_{j=0}^{(q-1)(r-1)} \sum_{k=0}^{p-1}a_jA^{j+kqr}.
\end{align*}

If $A^{qr}\neq 1$, then 
\begin{equation*} 
\sum_{k=0}^{p-1}a_jA^{j+kqr}=\frac{A^{j+pqr}-A^j}{A^{qr}-1}=0
\end{equation*} 
and therefore
\begin{equation*}  
\sum_{j=0}^{pqr-r-q+1}a_jA^j=0.
\end{equation*}

Henceforth we assume that $A^{qr}=1$. 
Then 
\begin{equation*}  
\sum_{j=0}^{pr-1-r+1}a_jA^j= \sum_{j=0}^{(q-1)(r-1)} \sum_{k=0}^{p-1}a_jA^{j+kqr} =p\sum_{j=0}^{(q-1)(r-1)}a_jA^j,
\end{equation*} 
and it remains to prove that
\begin{equation}\label{51}
\sum_{j=0}^{(q-1)(r-1)}a_jA^j=0.
\end{equation} 

It follows from the definition of $(a_j)$ that
\begin{equation*}
\sum_{j=0}^{(q-1)(r-1)}a_jA^j=\sum_{\substack{0\le j\le (q-1)(r-1)\\ j=\ell q+mr}}A^j-\sum_{\substack{0\le j\le (q-1)(r-1)\\ j=\ell q+mr+1}}A^j.
\end{equation*}
Applying Lemma \ref{l51} (iii) this yields
\begin{equation}\label{52}
\sum_{j=0}^{(q-1)(r-1)}a_jA^j=A^{(q-1)(r-1)}+(1-A)\sum_{\substack{0\le j<(q-1)(r-1)\\ j=\ell q+mr}}A^j.
\end{equation} 
We have to prove that this last expression vanishes. 

In order to simplify the last sum, for each $\ell=-1,0,\ldots, r-2$ we denote by $m_{\ell}$ the smallest integer $m$ satisfying $\ell q+mr\ge (q-1)(r-1)$. 
For example, we have $m_{-1}=q$ and thus $-q+m_{-1}r=q(r-1)$ because 
\begin{equation*}
-q+qr\ge (q-1)(r-1)\quad\text{but}\quad -q+(q-1)r<(q-1)(r-1).
\end{equation*}

Using this notation we deduce from \eqref{52} the following equality by summing some geometric series\footnote{The value $\ell=-1$ is not used here but we will need it later.}:
\begin{align*}
\sum_{j=0}^{(q-1)(r-1)}a_jA^j
&=A^{(q-1)(r-1)}+(1-A)\sum_{\ell=0}^{r-2}\sum_{m=0}^{m_{\ell}-1}A^{\ell q+mr}\\
&=A^{(q-1)(r-1)}+\frac{1-A}{1-A^r}\sum_{\ell=0}^{r-2}\left( A^{\ell q}-A^{\ell q+m_{\ell}r}\right) \\
&=A^{(q-1)(r-1)}+\frac{(1-A)(1-A^{(r-1)q})}{(1-A^r)(1-A^q)}\\ 
&\qquad\qquad -\frac{1-A}{1-A^r}\sum_{\ell=0}^{r-2}A^{\ell q+m_{\ell}r}.
\end{align*}

We claim that 
\begin{equation}\label{53}
\sum_{\ell=0}^{r-2}A^{\ell q+m_{\ell}r}=\sum_{k=0}^{r-2}A^{(q-1)(r-1)+k},
\end{equation} 
and hence 
\begin{multline}\label{54}
\sum_{j=0}^{(q-1)(r-1)}a_jA^j=A^{(q-1)(r-1)}+\frac{(1-A)(1-A^{(r-1)q})}{(1-A^r)(1-A^q)}\\
-\frac{1-A}{1-A^r}A^{(q-1)(r-1)}\frac{1-A^{r-1}}{1-A}.
\end{multline} 

\begin{proof}[Proof of \eqref{53}]
Since $(q,r)=1$, $\set{\ell q+m_{\ell}r\ :\ \ell=-1,0,\ldots, r-2}$ is a complete system of residues modulo $r$. Since, furthermore
\begin{equation*}
(q-1)(r-1)\le \ell q+m_{\ell}r<(q-1)(r-1)+r
\end{equation*}
by definition of $m_{\ell}$, we have
\begin{multline*}
\set{\ell q+m_{\ell}r\ :\ \ell=-1,0,\ldots, r-2}\\
=\set{(q-1)(r-1)+k\ :\ k=0,\ldots, r-1}
\end{multline*}
Finally, since we already know that $-q+m_{-1}r=q(r-1)$, we conclude that 
\begin{multline*}
\set{\ell q+m_{\ell}r\ :\ \ell=0,\ldots, r-2}\\
=\set{(q-1)(r-1)+k\ :\ k=0,\ldots, r-2},
\end{multline*} 
and this is equivalent to \eqref{53}. 
\end{proof}

In order to prove \eqref{51} it remains to show that the right side of \eqref{54} vanishes. Multiplying by $(1-A^r)(1-A^q)$ in order to eliminate the denominators, we are going to show, equivalently that the following expression vanishes:
\begin{multline*}
E:=(1-A^r)(1-A^q)A^{(q-1)(r-1)}+(1-A)(1-A^{(r-1)q})\\
-(1-A^{r-1})(1-A^q)A^{(q-1)(r-1)}.
\end{multline*}
Combining the terms containing the factor $A^{(q-1)(r-1)}$ and taking into account that $A^{qr}=1$, we may simplify the expression as follows:
\begin{align*}
E
&=A^{(q-1)(r-1)}(1-A^q)(A^{r-1}-A^r)+(1-A)(1-A^{(r-1)q})\\
&=A^{1-q-r}(1-A^q)A^{r-1}(1-A)+(1-A)(1-A^{-q})\\
&=A^{-q}(1-A^q)(1-A)+(1-A)(1-A^{-q})\\
&=(A^{-q}-1)(1-A)+(1-A)(1-A^{-q})\\
&=0.
\end{align*}
The proof of the theorem is complete.
\end{proof}

We end this section by giving an interesting corollary of Lemma \ref{l51}:

\begin{corollary}\label{c52}
The function $h_3$ is symmetrical with respect to the middle of its support.
\end{corollary}

\begin{proof}
Since each function $f\in X_T=X_{T_1}+X_{T_2}+X_{T_3}$ is symmetrical with respect to $T/2$, $h_3(T-t)$ is also orthogonal to $X_T$ and hence $h_3(T-t)\equiv Kh_3(t)$ for a suitable constant $K$. 

Since
\begin{equation*}
a_{pqr-q-r+1}=a_{qr-q-r+1}=a_{(q-1)(r-1)}=1=a_0
\end{equation*}
by Theorem \ref{t12} and Lemma \ref{l51}, $h_3$ has the same values at the beginning and at the end of its support, so that $K=1$.
\end{proof}

\section{Proof of the second part of Theorem \ref{t11}}\label{s6}

The proof is based on the study of the finer structure of $h_2$ and $(a_j)$. 
We need a lemma.

\begin{lemma}\label{l61}
If $\lnko{T_1,T_2}<T_2<T_1$, then setting $d:=\lnko{T_1,T_2}$ and $m:=T_2/d(\ge 2)$ we have 
\begin{equation*}
h_2\sim 1^d 2^d\cdots (m-1)^d m^{T_1-T_2+d} (m-1)^d\cdots 2^d 1^d,
\end{equation*}
where a block $a^b$ means $b$ consecutive elements $a$, and we omit the commas on the right side.
\end{lemma}

\begin{proof}
It follows from the formulas
\begin{equation*}
h_1\sim 1^{T_1}\quad \text{and} \quad h_2(t)=\sum_{j=0}^{m-1}h_1(t-dj)
\end{equation*}
that
\begin{equation*}
h_2\sim 1^d, 2^d,\ldots, (n-1)^d, n^b, (n-1)^d\ldots, 2^d, 1^d
\end{equation*}
for some integers $1\le n\le m$ and $b\ge 0$. 

We have $n=m$ because $T_1>T_2=md>(m-1)d$, and therefore the supports of the $n$ functions $h_1(t-dj)$ contain a common interval.
More precisely, the intersection of the supports is the interval $[(m-1)d,T_1]=[T_2-d,T_1]$, whence $b=T_1-T_2+d$.
\end{proof}

\begin{example}
If $(T_1,T_2)=(7,4)$ then $h_1=1^7$ and 
\begin{align*}
h_2(t) 
&= \sum_{j=0}^3h_1(t-j)\\
&\sim 1^7000+01^700+001^70+0001^7\\
&= 1234444321\\
&= 1^12^13^14^{(7-4+1)}3^12^11^1.
\end{align*}
\end{example}

\begin{proof}[Proof of Theorem 1.1 (ii)]
We distinguish  three cases:
\begin{enumerate}[\upshape (i)]
\item  If $\lnko{T_i,T_j}=\min\set{T_i,T_j}$ for some $i\ne j$, then $h_3$ is strictly positive.

Indeed, since $h_3$ does not change for any permutation of the periods, we may assume that 
\begin{equation*} 
\lnko{T_1,T_3}= \min(T_1,T_3)=T_1.
\end{equation*} 
Then $T_3$ is a multiple of $T_1$, so that $h_3=h_2$.
We conclude by recalling that $h_2$ is always strictly positive.

\item If $\lnko{T_i,T_j}=1$ for some $i\ne j$, then $h_3$ is strictly positive.

Indeed, we may assume by a permutation that $q=1$ where we use the notations 
\begin{equation*}  
q=\lnko{T_1,T_3} \quad , r=\lnko{T_2,T_3} \quad\text{and} \quad T_3=pqr 
\end{equation*} 
of Section \ref{s3}.
Applying Theorem \ref{t12} we have 
\begin{equation*}
h_3=\sum_{j=0}^{pr-r}a_jh_2(t-j)
\end{equation*} 
with $(a_j)=(1,0^{r-1})^{\infty}$. 
Since $h_2$ is strictly positive and all coefficients $a_j$ are nonnegative, we conclude that $h_3>0$.

\item Assume finally that $1<\lnko{T_i,T_j}<\min\set{T_i,T_j}$ whenever $i\ne j$. 
We will show that $h_3$ vanishes in the interval $[1,2)$. 

We use the usual expression 
\begin{equation}\label{61}
h_3(t) = \sum_{j=0}^{pqr-q-r+1}a_jh_2(t-j)
\end{equation}
where $p,q,r$ are given as in the preceding case. 
Since $q>1$ and $r>1$ by our asumptions, it follows from Theorem \ref{t12} that $a_0=1$ and $a_1=-1$. 
Furthermore, it follows from Lemma \ref{l61} that $h_2=1$ on the interval $(0,d)$ with $d=\lnko{T_1,T_2}\ge 2$. 
Applying the formula \eqref{61} we conclude that $h_3=a_0=1$ in $[0,1)$ and $h_3=a_0+a_1=1-1=0$ in $[1,2)$.
\end{enumerate}
\end{proof}

\section{Some results for more than three periods}\label{s7}

Given any finite number of positive integers $T_1,\ldots, T_n$, the set
\begin{equation*}
X:=\overline{X_{T_1}+\cdots+X_{T_n}}
\end{equation*} 
is spanned by the functions $e^{2i\pi mt/(T_1\cdots T_n)}$ where $m$ runs over the non-zero multiples of the numbers $T_1\cdots T_n/T_i$, $i=1,\ldots,n$. 
There exists again a critical length $T\ge\min\set{T_1,\ldots, T_n}$ and a function $h_n$ such that $X$ is the orthogonal complement of $h_n$ in $L^2(0,T)$. 

The method of Section \ref{s4} may be easily generalized to determine the function $h_n$. 
Let $n=4$ and $\lnko{T_1,T_2,T_3,T_4}$ for simplicity. 
Then we have 
\begin{equation*}
h_4(t)=\sum_{j=0}^Ta_jh_3(t-j)
\end{equation*}
with a suitable sequence $(a_j)$ of real numbers determined apart from a multiplicative constant.
Now the method of Section \ref{s4} shows that the numbers $a_j$ are the coefficients of the polynomial
\begin{equation*}
P_4(x)=\frac{(1-x^{T_1})(1-x^{T_2})(1-x^{T_3})(1-X^{\lnko{T_{1},T_{2},T_{3}}})Q_4(x)}{(1-x)(1-x^{\lnko{T_1,T_2}})(1-x^{\lnko{T_2,T_3}})(1-x^{\lnko{T_1,T_3}})}
\end{equation*}
with 
\begin{equation*}
Q_4(x)=\frac{(1-x^{T_4})(1-x^{\lnko{T_1,T_2,T_4}})(1-x^{\lnko{T_1,T_3,T_4}})(1-x^{\lnko{T_2,T_3,T_4}})}{(1-x^{\lnko{T_1,T_4}})(1-x^{\lnko{T_2,T_4}})(1-x^{\lnko{T_3,T_4}})(1-x^{\lnko{T_1,T_2,T_3,T_4}})}.
\end{equation*}
Using this construction we may prove the following

\begin{proposition}\label{p71}
If there exist two subsets $\{T_1^1,T_2^1,T_3^1\}$, $\{T_1^2,T_2^2,T_3^2\}$ of three elements and two subsets  $\{T_1^{3},T_2^{3}\}$ and $\{T_1^{4},T_2^{4}\}$ of two elements of $\{T_1,T_2,T_3,T_4 \}$ such that
\begin{equation}
\lnko{T_1^1,T_2^1,T_3^1}=\lnko{T_1^{3},T_2^{3}} \quad \mbox{and} \quad \lnko{T_1^2,T_2^2,T_3^2}=\lnko{T_1^{4},T_2^{4}},
\end{equation}
then the function $h_4(t)$ is nonnegative.
\end{proposition}

\begin{remarks}
~
\begin{enumerate}[\upshape (i)]
\item It can be shown that the hypothesis \eqref{p71} is satisfied if there exist two different couples $(i,j)$ and $(k,l)$ such that 
\begin{equation*}
\lnko{T_i,T_j}=\lnko{T_k,T_l}=1.
\end{equation*}
\item The hypothesis \eqref{p71} is also satisfied for
\begin{equation*}
(T_1,T_2,T_3,T_4)=(66,21,12,10).
\end{equation*}
In fact, we have even $h_4>0$ by a direct computation. 
This case was not covered by the theorems in \cite{Kom1987}.
\end{enumerate}
\end{remarks}

We have also a sufficient condition for $h_4$ to change sign:

\begin{proposition}\label{p72}
Let $(T_1,T_2,T_3,T_4)=(abc,abd,acd,bcd)$ with four relatively prime integers $1<a<b<c<d$. 
Then the corresponding orthogonal function $h_4$ has both positive and negative values in $(0,T)$.
\end{proposition}

\begin{remark}
The hypotheses of the proposition imply that $\lnko{T_i,T_j}>1$ whenever $i\ne j$. 
This last condition is not sufficient for $h_4$ to change sign.
For example, if $(T_1,T_2,T_3,T_4)=(12,18,39,42)$, then the corresponding function $h_4$ is nonnegative on $(0,T)$, although $\lnko{T_i,T_j}>1$ whenever $i\ne j$.
\end{remark}

\begin{proof}
In this case the polynomial $P_4$ is given by 
\begin{equation*}
P_4=P_0\frac{(1-X^c)(1-X^d)(1-X^{abc})(1-X^{abd})(1-X^{acd})(1-X^{bcd})}{(1-X^{ac})(1-X^{ad})(1-X^{bc})(1-X^{bd})(1-X^{cd})}
\end{equation*}
with
\begin{equation*}
P_0=\frac{(1-X^a)(1-X^b}{(1-X)(1-X)(1-X^{ab})}.
\end{equation*}
Since the power series $\sum a_kX^k$ starts with $a_0=1$, it is sufficient to prove that $a_c<0$ or $a_{b+c-1}<0$. 
By uniqueness we may write 
\begin{align*}
P_4(X)&=P_0(1-X^c)(1-X^d)(1-X^{abc})(1-X^{abd})(1-X^{acd})(1-X^{bcd})\\
& \times(1+X^{ab}+\cdots)(1+X^{ac}+\cdots)(1+X^{ad}+\cdots)(1+X^{bc}+\cdots)\\
&\times(1+X^{bd}+\cdots)(1+X^{cd}+\cdots).
\end{align*}
Since $ac,ad,bd,cd>b+c$, we have 
\begin{equation*}
P_4(X)\equiv P_0(1-X^c)(1-X^d)\equiv P_0(1-X^c-X^d) 
\end{equation*}
modulo $X^{b+c}$. 
Writing $P_0=\sum_kb_kX^k$  we obtain that
\begin{equation}\label{72} 
a_k=b_k-b_{k-c}-b_{k-d}. 
\end{equation}
It follows that
\begin{align*}
P_0&=\sum_kb_kX^k\\
&=(1+X+\dots+X^{a-1})(1+X+\dots+X^{b-1})(1+X^{ab}+X^{2ab}+\cdots)\\
&=(\sum_{k=0}^{a-1}(k+1)X^k+\sum_{k=a}^{b-2}aX^k+\sum_{k=b-1}^{a+b-2}X^k)(1+X^{ab}+X^{2ab}+\cdots).
\end{align*}
All coefficients $b_k$ of $P_0$ belong to $[\![0,a]\!]:=\{0,1,\dots,a\}$. More precisely, since $a+b-2<ab$, denoting by $r$ the remainder of $k$ modulo $ab$, we have 
\begin{align*}
b_k &=a \Leftrightarrow r\in [\![a-1,b-1]\!]\\ \intertext{and}
b_k &=0 \Leftrightarrow r\in  [\![a+b-1,ab-1]\!].
\end{align*}
We complete the proof by showing that $a_{b+c-1}\geq 0$ implies $a_c<0$. By \eqref{72} we have 
\begin{equation*}
a_{b+c-1}=b_{b+c-1}-b_{b-1}-b_{b+c-d-1}=b_{b+c-1}-a-b_{b+c-d-1}.
\end{equation*}
Hence $b_{b+c-1}=a$, the remainder $r$ of $b+c-1$ (mod $ab$) belongs to $[\![a-1,b-1]\!]$ and therefore 
\begin{equation*}
b+c-1=kab+r \quad \mbox{with} \quad r\in [\![a-1,b-1]\!].
\end{equation*}
Setting $r':=ab+r-b+1$ this may be rewritten in the form 
\begin{equation*}
c=(k-1)ab+r' \quad \mbox{with} \quad r'\in [\![ab+a-b,ab]\!].
\end{equation*}
Since $c,a$ and $b$ are relatively prime we cannot have $r'=ab$.
Therefore in fact $r'\in [\![ab+a-b,ab-1]\!]$, and hence $b_c=0$. 
Using \eqref{p72} again we conclude that 
\begin{equation*}
a_c=b_c-b_0-b_{c-d}=0-1=-1.\qedhere
\end{equation*}
\end{proof}

\subsection*{Acknowledgments}
The author wants to express his gratitude to Vilmos Komornik for his useful comments and suggestions and   Peter Mueller for his help in the proof of Proposition \ref{p72}. 
He also thanks Mohamad Maassarani for fruitful discussions.
%%%%%%%%%%%%%%%%%%%%%%%%%%%%%%%%%%%%%

\end{document}